\documentclass[reqno]{smfart}
\usepackage[english,frenchb]{babel}
\usepackage[T1]{fontenc}
\usepackage[latin1]{inputenc}
\usepackage{smfthm,smfenum,amsmath,amssymb,mathdots}

\usepackage{graphicx}

\NumberTheoremsIn{section}
\newtheorem{theoa}{Théorème}

\AtBeginDocument{\renewcommand{\labelitemi}{\textbullet}}

\DeclareMathOperator{\id}{id}

\DeclareMathOperator{\GL}{\sf GL}
\DeclareMathOperator{\PGL}{\sf PGL}

\DeclareMathOperator{\SO}{\sf SO}
\DeclareMathOperator{\Aut}{\sf Aut}

\DeclareMathOperator{\Diff}{\sf Diff}

\DeclareMathOperator{\h}{\sf h_{top}}

\DeclareMathOperator{\Fat}{\sf Fatou}

\DeclareMathOperator{\Ind}{\sf Ind}

\renewcommand\epsilon{\varepsilon}

\newcommand\N{\mathbf{N}}
\newcommand\Z{\mathbf{Z}}
\newcommand\Q{\mathbf{Q}}
\newcommand\R{\mathbf{R}}
\newcommand\C{\mathbf{C}}

\renewcommand\P{\mathbb{P}}

\renewcommand{\S}{\mathbb{S}}
\newcommand\T{\mathbb{T}}

\renewcommand{\O}{\mathcal O}
\newcommand{\Rot}{\mathrm{Rot}}
\newcommand{\e}{\mathrm{e}}

\begin{document}

\title[Sur la dynamique des difféomorphismes birationnels des surfaces]{Sur la dynamique des difféomorphismes birationnels des surfaces algébriques réelles : ensemble de Fatou et lieu réel}
\author{Arnaud Moncet}
\address{Université de Rennes 1\\
IRMAR\\
campus de Beaulieu\\
bâtiment 22--23\\
263~avenue du général Leclerc\\
CS 74205\\
35042 Rennes cedex
}
\email{moncet.arnaud@gmail.com}
\date{\today}

\begin{abstract}
On s'intéresse aux difféomorphismes birationnels des surfaces algébri\-ques réelles qui possèdent une dynamique réelle simple et une dynamique complexe riche. On donne un exemple d'une telle transformation sur $\P^1\times\P^1$, mais on montre qu'une telle situation est exceptionnelle et impose des conditions fortes à la fois sur la topologie du lieu réel et sur la dynamique réelle.
\end{abstract}

\begin{altabstract}
This text deals with birationnal diffeomorphisms of real algebraic surfaces which have simple real dynamics and rich complex dynamics. We give an example of such a transformation on $\P^1\times\P^1$, then we show that this situation is exceptional and implies strong conditions on both the topology of the real locus and the real dynamics.
\end{altabstract}

\keywords{Dynamique complexe, dynamique réelle, surfaces, transformations birationnelles, ensemble de Fatou, degré dynamique, domaine de rotation}

\maketitle

\section*{Introduction}

Dans \cite{bedford-kim-maxent}, Bedford et Kim ont exhibé un exemple d'automorphisme d'une surface algébrique réelle $X$, construite par éclatements de $\P^2$, qui est d'entropie maximale, c'est-à-dire que l'entropie topologique vue comme transformation sur $X(\R)$ est égale à celle sur $X(\C)$, et celle-ci est strictement positive. En termes qualitatifs, cela signifie que toute la richesse de la dynamique complexe est contenue dans le lieu réel. Par exemple, la maximalité de l'entropie implique que tous les points périodiques hyperboliques sont réels, sauf éventuellement ceux qui sont contenus dans une courbe rationnelle périodique (voir \cite{cantat-survey}).

De manière inverse, on peut se demander si on peut avoir une dynamique complexe non triviale (entropie strictement positive) avec une dynamique réelle relativement simple. Pour préciser ce dernier point, on peut regarder l'entropie sur le lieu réel, comme par exemple dans \cite[\textsection 5.3]{volumes}. Ici, on s'intéresse à une condition plus forte que la nullité de l'entropie sur $X(\R)$, qui est l'inclusion du lieu réel dans l'ensemble de Fatou. On donne une réponse  positive à la question lorsque l'on étend l'étude non pas aux automorphismes (qui sont des morphismes biréguliers sur $X(\C)$), mais aux difféomorphismes birationnels.

\begin{defi}
Soit $X$ une variété algébrique réelle (projective lisse). Un \emph{difféo\-morphisme birationnel} de $X$ est une application birationnelle $f:X(\C)\dashrightarrow X(\C)$ qui préserve la structure réelle et telle que $f$ est birégulier en restriction à $X(\R)$.
\end{defi}

La restriction à $X(\R)$ est alors un difféomorphisme (réel-analytique), d'où leur nom. Notons que lorsque $X(\R)$ est non vide, la restriction de $f$ à $X(\R)$ détermine~$f$ de manière unique. Dans \cite{kollar-mangolte}, les auteurs montrent que tout difféomorphisme de~$X(\R)$ peut être approché par des difféomorphismes birationnels lorsque $X$ est une surface birationnelle à $\P^2$ (sur $\R$), mais qu'en revanche ce n'est pas le cas pour la plupart des autres surfaces algébriques réelles.

Soit $f$ une application biméromorphe d'une variété complexe compacte $X$. On note $\Ind(f)$ son lieu d'indétermina\-tion : c'est une sous-variété analytique complexe de codimension au plus $2$, donc constituée d'un nombre fini de points lorsque $X$ est une surface. Son ensemble de Fatou correspond au plus grand ouvert sur lequel la dynamique est \og non-chaotique\fg ; son complémentaire est l'ensemble de Julia. Il est moins aisé à définir pour les applications méromor\-phes que pour les applications holomorphes ; néanmoins on peut prendre la définition suivante.

\begin{defi}
Soit ${f:X\dashrightarrow X}$ une application biméromorphe d'une variété complexe compacte. L'\emph{ensemble de Fatou} de $f$, noté $\Fat(f)$, est le plus grand ouvert~$U$ tel que les deux conditions suivantes soient satisfaites :
\begin{enumerate}
\item $U\cap\Ind(f^n)=\emptyset$ pour tout $n\in\Z$ ;
\item la famille ${(f^n_{|U})}_{n\in\Z}$ est une famille normale, c'est-à-dire que toute suite à valeurs dans $\{f^n\,|\,n\in\Z\}$ possède une sous-suite qui converge uniformément sur les compacts de $U$.\footnote{Notez que l'on considère ici les itérés positifs et négatifs de $f$.}
\end{enumerate}
Les composantes connexes de $\Fat(f)$ sont appelées \emph{composantes de Fatou}.
\end{defi}

On rappelle aussi que le (plus grand) degré dynamique d'une application méromor\-phe d'une variété kählérienne compacte $f:X\dashrightarrow X$ est défini par la formule
\begin{equation}
\lambda(f):=\lim_{n\to+\infty}\left\|{(f^n)}^*\right\|^{1/n},
\end{equation}
où $(f^n)^*$ désigne l'action induite par $f^n$ sur la cohomologie de $X$. D'après \cite{gromov-entropie} et \cite{yomdin}, ce degré dynamique est l'exponentielle de l'entropie topologique lorsque~$f$ est holomorphe. Dans le cas méromorphe, on a seulement l'inégalité $\h(f)\leq\log\lambda(f)$, et l'égalité est conjecturée sous certaines hypothèses (voir \cite{friedland,ds-entropy,guedj-entropie}).

Avec ces définitions, le résultat annoncé, qui est démontré dans la partie \ref{section-exbir}, est le suivant.

\begin{theoa}\label{theoa-birdiff}
Il existe un difféomorphisme birationnel $f$ sur la surface algébrique réelle $X=\P^1\times\P^1$, dont le lieu réel est homéomorphe à un tore $\T^2$, telle que
\begin{enumerate}
\item Le degré dynamique $\lambda(f)$ est strictement supérieur à $1$.
\item $X(\R)$ est inclus dans l'ensemble de Fatou.
\end{enumerate}
\end{theoa}

Plus précisément, on montre que $X(\R)$ possède un voisinage (dans $X(\C)$) sur lequel~$f$ est biholomorphe et conjugué à une rotation sur un produit de couronnes dans~$\C^2$. Cette construction est inspirée de celle des anneaux de Herman à l'aide de produits de Blaschke (voir par exemple \cite[\textsection 15]{milnor}), qui produit des endomorphismes de $\P^1$ tels que le lieu réel est contenu dans un domaine sur lequel $f$ agit comme une rotation. La condition sur le degré dynamique est l'analogue du fait que les endomorphismes de Herman sont de degré plus grand que $2$.

En revanche, on montre dans la partie \ref{section-fatou-reel} que cette situation est exceptionnelle, et impose des contraintes topologiques fortes sur les composantes connexes de $X(\R)$ et la dynamique réelle de~$f$.

\begin{theoa}\label{car-neg}
Soit $f$ un difféomorphisme birationnel d'ordre infini d'une surface algé\-brique réelle $X$ (projective lisse). Supposons qu'une composante connexe $S$ de~$X(\R)$ soit contenue dans $\Fat(f)$. Alors :
\begin{enumerate}
\item Le nombre de points périodiques sur $S$ est égal à la caractéristique d'Euler topologique $\chi(S)$.
\item En particulier $\chi(S)\geq 0$, si bien que topologiquement $S$ est une sphère, un tore, un plan projectif, ou une bouteille de Klein.
\item Un itéré de $f$ est conjugué à une \og rotation\fg~sur l'une de ces surfaces.
\end{enumerate}
\end{theoa}

\begin{coro}
Soit $f$ un difféomorphisme birationnel d'une surface algébrique réelle $X$, et soit $S$ une composante connexe de $X(\R)$. On suppose que $f$ possède au moins trois points périodiques sur $S$. Alors $S$ n'est pas contenue dans l'ensemble de Fatou.
\end{coro}

\begin{rema}
Ce théorème est vérifié de manière triviale en dimension $1$ pour les automorphismes, car les seules possibilités pour que $f$ soit d'ordre infini et d'ensemble de Fatou contenant $X(\R)\neq\emptyset$ sont les suivantes :
\begin{enumerate}
\item $X\simeq\P^1$ et $f\in\PGL_2(\R)$ est donné par une matrice de rotation.
\item $X(\C)$ est un tore $\C/\Lambda$ et un itéré de $f$ est une translation réelle sur ce tore.
\end{enumerate}
\end{rema}

\begin{rema}
Le fait que $X$ soit algébrique n'intervient pas dans la démonstra\-tion du théorème \ref{car-neg}. Il suffit de prendre une surface complexe compacte $X$ munie d'une structure réelle (c'est-à-dire une involution anti-holomorphe $\sigma:X\to X$), et de supposer que $f$ est biméromorphe, préserve la structure réelle et n'a pas de point d'indétermination sur $X(\R)={\sf Fix}(\sigma)$.
\end{rema}

\begin{enonce*}[remark]{Questions ouvertes}
\begin{enumerate}
\item
Existe-t-il des exemples comme celui du théorème~\ref{theoa-birdiff} avec pour $X(\R)$ une sphère, un plan projectif ou une bouteille de Klein ? Notons qu'un exemple sur une sphère (resp. sur un plan projectif) suffirait pour avoir les deux autres (resp. la bouteille de Klein), par éclatement de point(s) fixe(s).
\item 
Existe-t-il des exemples comme celui du théorème \ref{theoa-birdiff} où $f$ est un automorphisme au lieu d'un difféomorphisme birationnel ?
\item
Que dire des endomorphismes de $\P^2$ ? Se peut-t-il aussi que le lieu réel soit inclus dans l'ensemble de Fatou ? Contrairement aux exemples de Herman en dimension 1, on a ici des contraintes topologiques, car le lieu réel n'est pas contractible dans $\P^2(\C)$. Par exemple, on ne peut pas avoir d'endomorphisme de degré topologique impair qui soit un difféomorphisme sur le lieu réel.
\item
Que peut-on dire en dimension supérieure ? Les techniques utilisées dans cet article permettent certainement d'obtenir des résultats plus ou moins similaires, en utilisant une version adaptée d'une formule de Lefschetz qui prenne en compte les sous-variétés de points fixes. Cependant, le manque d'exemple ne nous permet pas de dégager un énoncé précis.
\end{enumerate}
\end{enonce*}

\subsection*{Remerciements}
Je suis particulièrement reconnaissant envers mon directeur de thèse Serge Cantat pour m'avoir inspiré cette étude et guidé dans mes recherches. Merci aussi à Frédéric Mangolte pour des discussions intéressantes sur le sujet, et pour m'avoir invité à exposer ces résultats à Angers lors de la rencontre \emph{Fonctions régulues} en mars 2012.


\section{Un exemple de difféomorphisme birationnel du tore}\label{section-exbir}


\subsection{Conjugaison à une rotation}\label{paragraphe1.1}

On désigne par $\S^1$ l'ensemble des nombres complexes de module $1$, et par~$\T^2$ le tore $\S^1\times\S^1$. Pour $\theta=(\theta_1,\theta_2)\in\R^2$, on note $\Rot_\theta$ la rotation d'angle $\theta$ sur le tore $\T^2$, donnée par la formule
\begin{equation}
\Rot_\theta(x,y)=(\e^{i\theta_1}x,\e^{i\theta_2}y).
\end{equation}

Le théorème suivant est démontré par Herman dans \cite{herman-tore}, en reprenant des idées de Arnol$'$d \cite{arnold} et Moser \cite{moser} (voir aussi {\cite[A.2.2]{herman}}) :

\begin{theo}[Arnold, Moser, Herman]\label{theo-herman}
Soit $\alpha=(\alpha_1,\alpha_2)\in\R^2$ qui vérifie une \emph{condition diophantienne}, c'est-à-dire qu'il existe $C>0$ et~${\beta>0}$ tels que
\begin{equation}\label{diop}
\lvert k_1\alpha_1 + k_2\alpha_2 + 2\pi k_3 \rvert \geq \frac{C}{\left(\lVert k\rVert_\infty\right)^\beta} \quad\quad \forall k=(k_1,k_2,k_3)\in\Z^3\backslash\{0\}.
\end{equation}
Pour tout $\epsilon>0$, il existe alors un voisinage $\mathcal{U}_{\alpha,\epsilon}$ de $\Rot_\alpha$ dans le groupe $\Diff^\omega(\T^2)$ tel que
\begin{equation}\label{conj-herman}
\forall g\in\mathcal{U}_{\alpha,\epsilon},\,
\exists\theta\in\left]-\epsilon,\epsilon\right[^2,\,
\exists \psi\in\Diff^\omega(\T^2),\,
g = \Rot_\theta\circ\psi\circ \Rot_\alpha\circ\psi^{-1}.
\end{equation}
\end{theo}

À partir de maintenant, on se place sur la surface rationnelle réelle ${X=\P^1\times\P^1}$, dont le lieu réel $
{X(\R)=\P^1(\R)\times\P^1(\R)}
$
s'identifie au tore $\T^2=\S^1\times\S^1$ via la transformation de Cayley
\begin{equation}\begin{split}
\Psi : X(\R) & \longrightarrow \T^2 \\
(x,y) & \longmapsto \left( \frac{x-i}{x+i} , \frac{y-i}{y+i} \right).
\end{split}\end{equation}
Via cette transformation, la rotation $\Rot_\theta$ sur $\T^2$ correspond à un automorphisme de~$X$ que l'on note $R_\theta$, et qui est donné par la formule
\begin{equation}
R_\theta(x,y) = 
\left(
\frac{x+\tan(\theta_1/2)}{-x\tan(\theta_1/2)+1}
,
\frac{y+\tan(\theta_2/2)}{-y\tan(\theta_2/2)+1}
\right).
\end{equation}

Fixons un nombre $\alpha\in\left[0,2\pi\right]^2$ qui vérifie une condition diophantienne du type~$(\ref{diop})$, et soit $\mathcal{U}_{\alpha,\epsilon}$ le voisinage de $\Rot_\alpha$ dans $\Diff^\omega(\T^2)$ donné par le théorème de Herman--Arnold--Moser \ref{theo-herman}, où
\begin{equation}\label{condepsilon}
\epsilon = \max(|\pi-\alpha_1|,|\pi-\alpha_2|)>0.
\end{equation}
On note $\mathcal{V}_{\alpha}=\Psi^{-1}\mathcal{U}_{\alpha,\epsilon}\Psi$ le voisinage correspondant de $R_\alpha$ dans $\Diff^\omega(X(\R))$.

\begin{prop}\label{fatoucontientreel}
Soit $f$ un difféomorphisme birationnel de $X$ dont la restriction au lieu réel est dans $\mathcal{V}_{\alpha}$. Il existe alors ${\theta\in\left]-\epsilon,\epsilon\right[^2}$ tel que la transformation $R_\theta f$ soit analytiquement conjuguée, sur un voisinage $\Omega$ de $X(\R)$, à la rotation d'angle~$\alpha$ sur une bicouronne de $\C^2$. Autrement dit, il existe un difféomorphisme analytique complexe~${\psi:\mathcal{C}_1\times\mathcal{C}_2\overset{\sim}\rightarrow\Omega}$, où $\mathcal{C}_1$ et $\mathcal{C}_2$ sont des couronnes dans $\C$, tel que
\begin{equation}\label{form-conj-rot}
\psi^{-1}\circ(R_\theta f)\circ\psi\;(z_1,z_2)=({\rm e}^{i\alpha_1}z_1,{\rm e}^{i\alpha_2}z_2).
\end{equation}
En particulier $
X(\R)\subset\Omega\subset\Fat(R_\theta f).
$
\end{prop}

\begin{proof}
Par construction du voisinage $\mathcal{V}_{\alpha}$, il existe~$\theta\in\left]-\epsilon,\epsilon\right[^2$ et un difféomorphisme réel-analytique $\psi:\T^2\overset{\sim}\to X(\R)$ tels que 
\begin{equation}
f_{|X(\R)}=R_{-\theta}\circ\psi\circ \Rot_\alpha\circ\psi^{-1}.
\end{equation}
Comme $\psi$ est analytique, il se prolonge en un difféomorphisme analytique complexe, toujours noté $\psi$, d'un produit de couronnes
\begin{equation}
\mathcal{C}_1\times\mathcal{C}_2 = \left\{
(z_1,z_2)\in\C^2,\;-\eta_i<\log|z_i|<\eta_i\quad\forall i\in\{1,2\}
\right\}
\end{equation}
vers un voisinage $\Omega$ de $X(\R)$ dans $X(\C)$. On peut supposer, quitte à réduire les~${\eta_i>0}$, que $\Omega$ ne contient pas de point d'indétermination pour $f$. On a alors la formule $(\ref{form-conj-rot})$ par prolongement analytique sur $\mathcal{C}_1\times\mathcal{C}_2$.
\end{proof}

Comme les difféomorphismes birationnels sont denses dans les difféomorphismes de~$X(\R)$ d'après \cite{kollar-mangolte}, on peut s'attendre à ce qu'il existe de telles transformations birationnelles $R_\theta f$ qui soient de degré dynamique strictement supérieur à un. C'est ce que nous allons montrer maintenant.


\subsection[Construction explicite d'un difféomorphisme birationnel]{Construction explicite d'un difféomorphisme birationnel de grand de\-gré dynamique}

Le lemme suivant, que l'on ne redémontre pas ici, donne une condition suffisante pour avoir un grand degré dynamique.

\begin{lemm}[{\cite[th. 3.1]{xie}}]\label{theo-xie}
Soit $X$ une surface rationnelle complexe, et soit $f:X\dashrightarrow X$ une application birationnelle de $X$. On fixe un $\R$-diviseur ample $L$ sur~$X$, et on suppose que
\begin{equation}
\deg_L(f^2)> C \deg_L(f),
\end{equation}
où $C=2^{3/2}3^{18}$, et $\deg_L(g)$ désigne le nombre d'intersection $g^*L\cdot L$. Alors 
\begin{equation}
\lambda(f) > \frac{\deg_L(f^2)}{C\deg_L(f)} > 1.
\end{equation}
\end{lemm}

\begin{rema}
Lorsque $X=\P^2$ et $L$ est une droite, $\deg_L$ correspond au degré usuel d'une application rationnelle, c'est-à-dire le degré des polynômes homogènes $P$, $Q$ et $R$ tels que $f=[P:Q:R]$. Dans \cite{xie}, le lemme \ref{theo-xie} est démontré uniquement dans ce cadre, mais la démonstration qui en est faite se transpose naturellement au cas plus général donné ci-dessus. Nous allons l'appliquer avec $X=\P^1\times\P^1$ et $L$ un vecteur propre pour $f^*$ dans $H^2(X;\R)$\footnote{Contrairement à ce qui se passe pour les automorphismes, $f^*$ peut avoir des vecteurs propres qui sont des classes amples.}.
\end{rema}

Fixons $d$ un entier tel que 
\begin{equation}\label{condsurd}
d \geq \frac{\sqrt{C}}{2}=2^{-1/4} \times 3^9\approx 16551,4.
\end{equation}
Pour $n\in\N^*$, soit $F_n$ la fraction rationnelle suivante :
\begin{equation}
F_n(x)=\frac{x^{2d}+\frac{2}{n}x^d+1}{x^{2d}+1}.
\end{equation}
Ses zéros et ses pôles sont simples, et sont donnés par les ensembles
\begin{gather}
Z_n = \left\{
\e^{i(\pm\arccos(\frac{1}{n})+2k\pi)/d} \,\big|\, k\in\{0,\cdots,d-1\}
\right\},\\
P_n = \left\{
\e^{i(\pm\frac{\pi}{2}+2k\pi)/d} \,\big|\, k\in\{0,\cdots,d-1\}
\right\}.
\end{gather}
En particulier, $Z_n\cap P_n=\emptyset$ et $(Z_n\cup P_n)\subset\C\backslash\R$.
Soit $g_n:X\dashrightarrow X$ l'application rationnelle définie par
\begin{equation}
g_n(x,y)=(F_n(x)y,x).
\end{equation}
Cette application est birationnelle réelle, d'inverse $(x,y)\to(y,x/F_n(y))$. Les points d'indétermination de $g_n$ sont donnés par
\begin{equation}
\Ind(g_n) = \big(Z_n\times\{\infty\}\big) \cup \big(P_n\times\{0\}\big),
\end{equation}
et ceux de $g_n^{-1}$ sont donnés par
\begin{equation}
\Ind(g_n^{-1}) = (\{0\}\times Z_n) \cup (\{\infty\}\times P_n).
\end{equation}
En particulier, ces points d'indétermination ne sont pas réels, donc ${g_n}$ est un difféomor\-phisme birationnel de $X$. D'autre part, $\Ind(g_n)$ et $\Ind(g_n^{-1})$ ne s'intersectent pas, donc d'après \cite{diller-favre}, $(g_n^2)^*=(g_n^*)^2.$

\begin{theo}\label{fntheta}
Pour $n\in\N^*$ et $\theta=(\theta_1,\theta_2)\in\R^2$, on considère le difféomorphis\-me birationnel
\begin{equation}
f_{n,\theta}=R_\theta \circ g_n^2.
\end{equation}
On suppose que 
$\theta_j\neq \pi \mod 2\pi$ pour $j\in\{1,2\}$. Alors :
\begin{enumerate}
\item
$\Ind(f_{n,\theta})\cap\Ind(f_{n,\theta}^{-1})=\emptyset$, et donc $(f_{n,\theta}^2)^*=(f_{n,\theta}^*)^2$.
\item
$\lambda(f_{n,\theta})>1$.
\end{enumerate}
\end{theo}

\begin{proof}
Les points d'indétermination de $f_{n,\theta}$ sont donnés par
\begin{align}
\Ind(f_{n,\theta})&=\Ind(g_n^2)\\
&=\Ind(g_n)\cup g_n^{-1}\left(\Ind(g_n)\right)\\
&=\big(Z_n\times\{\infty\}\big)\cup\big(P_n\times\{0\}\big)\cup\big(\{\infty\}\times Z_n\big)\cup\big(\{0\}\times P_n\big),
\end{align}
et ceux de $f_n^{-1}$ par
\begin{align}
\Ind(f_{n,\theta}^{-1})&=R_\theta(\Ind(g_n^{-2}))\\
&=R_\theta(\Ind(g_n^{-1})\cup g_n(\Ind(g_n^{-1}))\\
&=R_\theta\big((\{0\}\times Z_n)\cup(\{\infty\}\times P_n)\cup(Z_n\times\{0\})\cup(P_n\times\{\infty\})\big)\\
\begin{split}
&=\big(\{t_{\theta_1}\}\times Z_{n,\theta_2}\big)\cup\big(\{-t_{\theta_1}^{-1}\}\times P_{n,\theta_2}\big)\\
&\quad\quad \cup\big(Z_{n,\theta_1}\times\{t_{\theta_2}\}\big)\cup\big(P_{n,\theta_1}\times\{-t_{\theta_2}^{-1}\}\big),
\end{split}
\end{align}
où $t_{\theta_j}=\tan(\theta_j/2)$, $Z_{n,\theta_j}=R_{\theta_j}(Z_n)0$, $P_{n,\theta_j}=R_{\theta_j}(P_n)$, et les $R_{\theta_j}$ sont les composantes de $R_\theta$.

La condition sur $\theta_j$ implique $t_{\theta_j}\neq\infty$. Comme de plus $t_{\theta_j}$ est réel, il n'est pas dans~$Z_n\cup P_n$. Ainsi, la seule possibilité pour que $\Ind(f_{n,\theta})$ et $\Ind(f_{n,\theta}^{-1})$ s'intersectent est d'avoir
\begin{equation}
t_{\theta_1}=0 \quad \text{et} \quad \big(Z_n\cap P_{n,\theta_2}\big)\cup\big(P_n\cap Z_{n,\theta_2}\big)\neq\emptyset,
\end{equation}
ou la même condition en inversant $\theta_1$ et $\theta_2$.

On note $\mathcal{R}$ le groupe $\{R_\theta\,|\,\theta\in\R\}$. Ce groupe agit sur $\P^1(\C)$, et l'orbite d'un point~$x\in\S^1\backslash\{i,-i\}$ est un cercle transverse à $\S^1$, qui n'intersecte $\S^1$ qu'aux points $x$ et $-1/x$. Les points $i$ et $-i$ sont quant à eux fixes par $\mathcal{R}$. Comme $Z_n$ et $P_n$ sont des sous-ensembles de $\S^1$ qui ne s'intersectent pas, et comme $P_n$ est stable par $x\mapsto-1/x$, on en déduit que $Z_n\cap(\mathcal{R}\cdot P_n)=\emptyset$. Par conséquent, les intersections ci-dessus sont toujours vides, et donc $\Ind(f_{n,\theta})\cap\Ind(f_{n,\theta}^{-1})=\emptyset$.

Notons $H=\P^1\times\{0\}$ et $V=\{0\}\times\P^1$ les diviseurs de $X$ dont les classes de Chern forment une base de $H^2(X(\C);\R)$. Dans cette base, la matrice de $g_n^*$ s'écrit
\begin{equation}\label{matriceA}
A=\begin{pmatrix}2d&1\\1&0\end{pmatrix}.
\end{equation}
En effet, il suffit de calculer les nombres d'intersection $g_n^*H\cdot V$, $g_n^*V\cdot V$, etc. qui correpondent aux degrés des fonctions coordonnées par rapport à $x$ et $y$.
Cette matrice admet $\lambda=d+\sqrt{d^2+1}$ pour plus grande valeur propre, et $L=\lambda H + V$ pour vecteur propre associé à $\lambda$. Notons que $L$ est une classe ample.

Comme $(g_n^2)^*=(g_n^*)^2$ et $R_\theta^*=\id$, la matrice de $f_{n,\theta}^*$ est la matrice $A^2$. De plus $(f_{n,\theta}^2)^*=(f_{n,\theta}^*)^2$, donc la matrice de $(f_{n,\theta}^2)^*$ est $A^4$. On en déduit, avec les notations du lemme \ref{theo-xie}, que $\deg_L(f_{n,\theta})=2\lambda^3$ et $\deg_L(f_{n,\theta}^2)=2\lambda^5$ (on a~${L^2=2\lambda}$). Ainsi,
\begin{equation}
\frac{\deg_L(f_{n,\theta}^2)}{\deg_L(f_{n,\theta})} = \lambda^2 \geq 4d^2 \geq 3^{18}\sqrt{2}
\end{equation}
d'après la condition $(\ref{condsurd})$. Le lemme \ref{theo-xie} implique alors $\lambda(f_{n,\theta})>1$.
\end{proof}


\subsection{Démonstration du théorème \ref{theoa-birdiff}}

Comme $g_n$ converge vers l'application $(x,y)\mapsto(y,x)$ dans $\Diff^\omega(X(\R))$, la suite $\left(f_{n,\alpha}\right)_{n\in\N^*}$ converge vers $R_\alpha$. Pour~$n$ suffisamment grand, les applications $f_{n,\alpha}$ sont donc dans le voisinage~$\mathcal{V}_{\alpha}$ de~$R_\alpha$ (défini au paragraphe~\ref{paragraphe1.1}). D'après la proposition \ref{fatoucontientreel}, il existe~${\theta\in\left]-\epsilon,\epsilon\right[^2}$ tel que $R_\theta f_{n,\alpha}=f_{n,\alpha+\theta}$ soit conjugué à une rotation dans un voisinage de $X(\R)$. D'après le choix de $\epsilon$ (cf.~$(\ref{condepsilon})$), la condition $\alpha_j+\theta_j\neq\pi \mod 2\pi$ du théorème \ref{fntheta} est satisfaite. On a donc :
\begin{align}
\lambda(f_{n,\alpha+\theta})>1
\quad \text{et} \quad
\Fat(f_{n,\alpha+\theta})\supset X(\R),
\end{align}
et le théorème \ref{theoa-birdiff} est démontré.
\qed


\section{Le cas général}\label{section-fatou-reel}

Soit $X$ une surface complexe compacte munie d'une structure réelle, et soit $f$ une application biméromorphe de $X$ compatible avec la structure réelle. On suppose que~$f$ est d'ordre infini et qu'une composante connexe $S$ de $X(\R)$ est contenue dans l'ensemble de Fatou (en particulier il n'y a pas de point d'indétermination sur $S$). Quitte à passer à un itéré, on peut supposer que $f(S)=S$. Nous allons montrer que :
\begin{enumerate}
\item $f$ possède exactement $\chi(S)$ points fixes sur $S$, quitte à le remplacer par un itéré ;
\item en particulier, $S$ est homéomorphe à une sphère, un tore, un plan projectif ou une bouteille de Klein ;
\item un itéré de $f$ est conjugué à une rotation sur $S$.
\end{enumerate}
Dans ce qui suit, on note $\Omega$ la composante de Fatou contenant $S$.


\subsection{Domaines de rotation}

Soit $\Omega$ une composante de Fatou telle que $f(\Omega)=\Omega$. Conformément à \cite{fs-higher1}, on dit que $\Omega$ est un \emph{domaine de rotation} (ou domaine de Siegel) lorsqu'il existe une suite ${m_k\to\pm\infty}$ telle que
\begin{equation}
f^{m_k}\mathop{\longrightarrow}\limits_{k\to+\infty}\id_\Omega
\end{equation}
uniformément sur les compacts de $\Omega$. On a la caractérisation suivante (voir aussi \cite[\textsection 1]{bedford-kim-rotation}) :

\begin{prop}\label{carac-rot}
Soit $\Omega$ une composante de Fatou fixe par $f$. On note $\mathcal{G}(\Omega)$ l'adhérence du sous-groupe engendré par $f$ dans $\Aut(\Omega)$, pour la topologie de la convergence uniforme sur les compacts. Les énoncés suivants sont équivalents :
\begin{enumerate}
\item $\Omega$ est un domaine de rotation pour $f$ ;
\item le groupe $\mathcal{G}(\Omega)$ est compact.
\end{enumerate}
Si ces conditions sont vérifiées, $\mathcal{G}(\Omega)$ est alors un groupe de Lie abélien compact, dont la composante connexe de l'identité~$\mathcal{G}(\Omega)^0$ est un tore réel $\T^d$. L'entier $d$ est appelé \emph{rang} du domaine de rotation.
\end{prop}

\begin{proof}
On peut adapter la preuve d'un théorème de H. Cartan (voir par exemple \cite[chapitre 5]{narasimhan}) pour montrer que le groupe $\mathcal{G}(\Omega)$ est localement compact, en utilisant que ses éléments forment une famille normale en tant qu'applications holomorphes de $\Omega$ dans $X(\C)$\footnote{L'énoncé de Cartan concerne les domaines bornés de $\C^n$ ; le fait de considérer ici une famille normale remplace le théorème de Montel.}. De plus, ce groupe agit par difféomorphismes de classe $\mathcal{C}^2$ sur $\Omega$, et tout élément qui agit trivialement sur un ouvert est l'élément neutre. Un théorème de Bochner et Montgomery \cite{bm1} implique alors que $\mathcal{G}(\Omega)$ est un groupe de Lie. Comme de plus il est abélien, il est isomorphe à $F\times\T^d\times\R^k$, où $F$ est un groupe abélien fini. Chacune des deux assertions est alors équivalente à~${k=0}$.
\end{proof}

En dimension $1$, les domaines de rotation sont les disques de Siegel et les anneaux de Herman, et le groupe $\mathcal{G}(\Omega)^0$ est un cercle (cf. \cite{milnor}). En dimension $2$, on peut montrer que le rang des domaines de rotation est $1$ ou~$2$ pour les automorphismes d'entropie positive (voir \cite[théorème 1.6]{bedford-kim-rotation}). Dans \cite{mcmullen-k3-siegel}, McMullen donne des exemples de tels automorphismes sur des surfaces K3 non algébriques qui admettent un domaine de rotation de rang $2$ (voir aussi \cite{oguiso}). Sur les surfaces rationnelles, il existe des exemples de domaines de rotation de rang $1$ et $2$ (voir \cite{bedford-kim-maxent} et \cite{mcmullen-rat}).

\begin{prop}
Soit $S$ une composante connexe de $X(\R)$ telle que $f(S)=S$. On suppose que $S\subset\Fat(f)$, et on note $\Omega$ la composante de Fatou contenant $S$. Alors $\Omega$ est un domaine de rotation.
\end{prop}

\begin{proof}
Comme $\Omega$ est contenu dans l'ensemble de Fatou, il existe une suite $n_k\to+\infty$ telle que $f^{n_k}\to g$ uniformément sur les compacts de $\Omega$, avec ${g:\Omega\to X(\C)}$ holomorphe. Quitte à passer à une sous-suite, on peut également supposer que ${f^{-n_k}\to h}$ et $f^{m_k}\to i$, avec $m_k:=n_{k+1}-n_k\to+\infty$. En restriction à $S$, la convergence est uniforme et les fonctions $g$, $h$ et $i$ sont à valeurs dans $S$. On peut donc composer les limites dans les expressions $\id_S=f^{n_k}\circ f^{-n_k}$ et $f^{m_k}=f^{n_{k+1}}\circ f^{-n_k}$, ce qui donne $\id_S=g_{|S}\circ h_{|S}$ et $i_{|S}=g_{|S}\circ h_{|S}$. En particulier, $i_{|S}=\id_S$, et par prolongement analytique on en déduit que $i=\id_\Omega$. Ainsi $\Omega$ est un domaine de rotation.
\end{proof}


\subsection{Linéarisation au voisinage d'un point fixe}

Supposons qu'il existe un point fixe $x$ de $f$ sur $S$. Un argument de linéarisation dû là encore à H. Cartan montre que le groupe des germes en $x$ d'automorphismes de $\mathcal{G}(\Omega)$ est conjugué à un sous-groupe de $\GL_2(\R)$. Comme ce sous-groupe est compact, abélien et infini, il est lui-même conjugué à $\SO_2(\R)$. On obtient ainsi le résultat suivant.

\begin{prop}\label{conj-rot}
La restriction de $f$ à $S$ est conjuguée à une rotation irrationnelle au voisinage de chaque point fixe $x\in S$. En particulier, ces points fixes sont isolés, et l'endomorphisme~${({\rm d}f(x)-\id)}$ sur l'espace tangent en~$x$ est inversible et de déterminant positif.
\end{prop}

\begin{proof}
On a déjà vu que $f_{|S}$ est conjugué à une rotation dans un voisinage de $x$. Si celle-ci était d'ordre fini $k$, alors $f^k$ serait l'identité sur un voisinage de $x$, donc sur $X$ tout entier par prolongmement analytique, contredisant ainsi l'hypothèse sur l'ordre de $f$. Comme une rotation irrationnelle n'a pas de point fixe autre que l'origine, ceci implique que les points fixes sur $S$ sont isolés. Au voisinage d'un tel point $x$, $({\rm d}f(x)-\id)$ est conjugué à une matrice de la forme
\begin{equation}
\begin{pmatrix}
\cos(\theta)-1 & -\sin(\theta)\\
\sin(\theta) & \cos(\theta)-1
\end{pmatrix}
\end{equation}
avec $\theta\in\R\backslash 2\pi\Q$, qui a pour déterminant $2-2\cos(\theta)>0$.
\end{proof}


\subsection{Démonstration du théorème \ref{car-neg}}

Quitte à prendre un itéré de $f$, on peut supposer que $f$ est dans la composante connexe de l'identité $\mathcal{G}(\Omega)^0$ du groupe de Lie compact $\mathcal{G}(\Omega)$. Comme $\mathcal{G}(\Omega)^0$ agit par difféomorphismes sur $S$, on en déduit que la restriction de $f$ à $S$ est isotope à l'identité. La formule des points fixes de Lefschetz (voir par exemple \cite{griffiths-harris}) donne alors, en vertu de la proposition \ref{conj-rot} :
\begin{equation}\label{lef}
\chi(S) = {\sf card}\{x\in S\,|\,f(x)=x\}.
\end{equation}
En particulier, on obtient $\chi(S)\geq 0$, donc $S$ est une sphère, un tore, un plan projectif ou une bouteille de Klein.

Par ailleurs, comme $\mathcal{G}(\Omega)$ agit fidèlement sur $S$, le domaine de rotation $\Omega$ est de rang $1$ ou $2$. Le groupe de Lie $\mathcal{G}(\Omega)^0$ est donc soit un cercle $\S^1=\left\{z\in\C\,\big|\,|z|=1\right\}$, soit un tore $\T^2=\S^1\times\S^1$.

Si $\mathcal{G}(\Omega)^0\simeq\T^2$, alors l'orbite d'un point générique $x_0\in S$ est un tore de dimension~$2$, et le difféomorphisme
\begin{equation}
\begin{split}
\psi : \mathcal{G}(\Omega)^0\simeq\T^2 &\longrightarrow S\\
g & \longmapsto g(x_0),
\end{split}
\end{equation}
conjugue l'action par translations de $\T^2$ sur lui-même à l'action de $\mathcal{G}(\Omega)^0$ sur~$S$. En particulier, $f_{|S}$ est conjugué à une rotation.

On considère maintenant le cas où $\mathcal{G}(\Omega)^0$ est le groupe $\S^1$. On utilise le lemme suivant, qui est sans doute bien connu, mais difficile à localiser dans la littérature. On en donne une démonstration en annexe.

\begin{lemm}\label{lemm-action}
Soit une action de classe $\mathcal{C}^1$, fidèle et sans point fixe, du groupe $\S^1$ sur le cylindre $\mathcal{C}=[0,1]\times\S^1$ ou sur le ruban de Möbius ${\mathcal{M}=\mathcal{C}/s}$, où $s$ est l'involution $(t,z)\mapsto(1-t,{\e}^{i\pi}z)$. Alors cette action est conjuguée à l'action standard de $\S^1$ par rotations :
\begin{equation}
{\rm e}^{i\theta}\cdot(t,z)=(t,{\rm e}^{i\theta}z).
\end{equation}
\end{lemm}

Si $S$ est orientable, la formule de Lefschetz $(\ref{lef})$ montre que $f$ possède $2$ points fixes sur $S$ si $S$ est une sphère, ou $0$ point fixe si $S$ est un tore. Dans le premier cas, $f$ est conjugué à une rotation dans des petits disques au voisinage de chaque point fixe ; en retirant ces disques, on obtient un cylindre $S'\subset S$ sur lequel $\S^1$ agit fidèlement et sans point fixe. Dans le second cas, en découpant le tore $S$ suivant une orbite de~$\S^1$ (celle-ci est une courbe fermée simple non contractible, car sinon il y aurait un point fixe dans le disque bordé par cette courbe), on obtient aussi un cylindre $S'$ sur lequel $\S^1$ agit fidèlement et sans point fixe. On sait d'après le lemme \ref{lemm-action} que l'action de $\S^1$ sur $S'$ est conjuguée à une rotation, et en recollant, on obtient une action par rotation sur $S$.

Il reste à traiter le cas où $S$ est non orientable. Si $S$ est un plan projectif,~le groupe $\S^1$ possède $\chi(S)=1$ point fixe sur $S$, autour duquel l'action est linéarisable et conjuguée à une rotation. En enlevant un petit disque autour de ce point fixe, on obtient un ruban de Möbius, sur lequel $f$ agit par rotations d'après le lemme \ref{lemm-action}. Enfin, si $S$ est une bouteille de Klein, on obtient de même un ruban de Möbius en découpant le long d'une orbite, et $f$ agit par rotations sur ce ruban. Le résultat s'en déduit par recollement.
\qed

\appendix

\section{Démonstration du lemme \ref{lemm-action}}

\begin{enumerate}
\item
On traite d'abord le cas où $\S^1$ agit sur le cylindre. Cette action provient d'un champ de vecteurs $\overrightarrow V$ de classe $\mathcal{C}^1$ sur $\mathcal{C}$. Comme l'action est sans point fixe, $\overrightarrow V$ ne s'annule pas. En particulier, les orbites sous l'action de~$\S^1$, qui sont les trajectoires du champ de vecteurs, sont des cercles homotopes aux bords du cylindre (c'est une conséquence facile du théorème de Poincaré--Bendixson).

En faisant tourner ce champ de vecteurs d'un angle positif de $\pi/2$, on obtient un champ de vecteurs orthogonal~$\overrightarrow W$. Chaque trajectoire $\gamma$ de $\overrightarrow V$ découpe $\mathcal C$ en deux cylindres $\mathcal{C}_1$ et $\mathcal{C}_2$ (éventuellement un des deux est un cercle si $\gamma$ est un des deux bords), avec le champ de vecteurs $\overrightarrow W$ qui est rentrant le long de~$\gamma$ sur~$\mathcal{C}_1$, et sortant sur $\mathcal{C}_2$. Ainsi, toute trajectoire de $\overrightarrow W$ sur $\mathcal{C}_1$ (resp. $\mathcal{C}_2$) partant d'un point de $\gamma$ (resp. terminant sur un point de $\gamma$) ne peut pas revenir sur le bord $\gamma$ une seconde fois, à cause de l'orientation de $\overrightarrow W$. On en déduit le fait suivant :
\begin{center}
Les trajectoires des champs $\overrightarrow V$ et $\overrightarrow W$ se coupent en au plus un point.
\end{center}

Considérons une trajectoire $\gamma$ pour $\overrightarrow W$, que l'on fait partir de l'un des deux bords du cylindre (celui pour lequel le champ de vecteurs $\overrightarrow W$ est rentrant). Comme cette trajectoire ne peut pas revenir sur le premier bord, a priori deux cas sont possibles :
\begin{enumerate}
\item\label{cas-a}
La trajectoire $\gamma$ atteint le deuxième bord du cylindre en un temps fini $t_1>0$.
\item
La courbe $\gamma$ est définie pour tout $t\geq 0$, et reste dans l'intérieur du cylindre. En vertu du théorème de Poincaré--Bendixson, cette trajectoire admet un point limite ou un cycle limite $\Gamma$. Comme le champ $\overrightarrow W$ ne s'annule pas, il ne peut s'agir que d'un cycle limite $\Gamma$, qui ne peut être homotopiquement trivial, donc qui est homotope au bord du cylindre. Or l'orbite sous $\S^1$ d'un point de $\Gamma$ est également un cercle homotope à $\Gamma$, qui coupe $\Gamma$ transversalement en au plus un point, d'après le fait énoncé plus haut : ceci est impossible.
\end{enumerate}
On est donc dans le cas (a). Quitte à renormaliser et à inverser la trajectoire, on a ainsi construit une courbe $\gamma:[0,1]\to \mathcal{C}$ telle que
\begin{itemize}
\item $\gamma(0)\in\{0\}\times\S^1$ et $\gamma(1)\in\{1\}\times\S^1$ ;
\item toute orbite coupe $\gamma$ exactement une fois, et de manière transverse.
\end{itemize}
Le difféomorphisme
\begin{equation}
\begin{split}
\psi : [0,1]\times\S^1 &\longrightarrow \mathcal{C}\\
(t,{\rm e}^{i\theta}) & \longmapsto {\rm e}^{i\theta}\cdot \gamma(t),
\end{split}
\end{equation}
conjugue alors l'action standard de $\S^1$ sur $[0,1]\times\S^1$ à notre action sur $\mathcal{C}$.

\item Considérons maintenant une action fidèle et sans point fixe de $\S^1$ sur le ruban de Möbius $\mathcal{M}$. On note $\pi:\mathcal{C}\to\mathcal{M}$ le revêtement double correspondant à l'involution~${s:(t,z)\mapsto(1-t,{\e}^{i\pi}z)}$ sur $\mathcal{C}$.

Pour des raisons topologiques, le bord $\partial\mathcal{M}$ est nécessairement une orbite sous l'action de $\S^1$. Toute orbite est homotope à ce bord, donc dans $\pi_*\left(\pi_1(\mathcal{C})\right)$. On en déduit que l'action de $\S^1$ sur $\mathcal{M}$ se relève en une action de $\S^1$ sur $\mathcal{C}$, nécessairement fidèle et sans point fixe (l'action est fidèle, car elle l'est en restriction à chacun des bords). D'après le premier point, l'espace des orbites pour cette dernière action s'identifie avec le segment $[0,1]$, et l'involution $s$ induit une involution (non triviale) sur les orbites : par continuité, il existe donc une orbite $\gamma$ qui est fixée par $s$.

En découpant $\mathcal{C}$ le long de cette orbite $\gamma$, on obtient deux cylindres identiques $\mathcal{C}_1$ et $\mathcal{C}_2$, et le ruban de Möbius $\mathcal{M}$ est réobtenu à partir d'un tel cylindre en recollant les deux moitiés de $\gamma$ (voir figure \ref{fig-mobius}). Il suffit maintenant d'appliquer le premier cas à l'un de ces cylindres $\mathcal{C}_i$, puis de recoller de la manière que nous venons d'indiquer pour obtenir la conjugaison voulue.\qed

\begin{figure}[h]
\begin{center}
\includegraphics[scale=0.6]{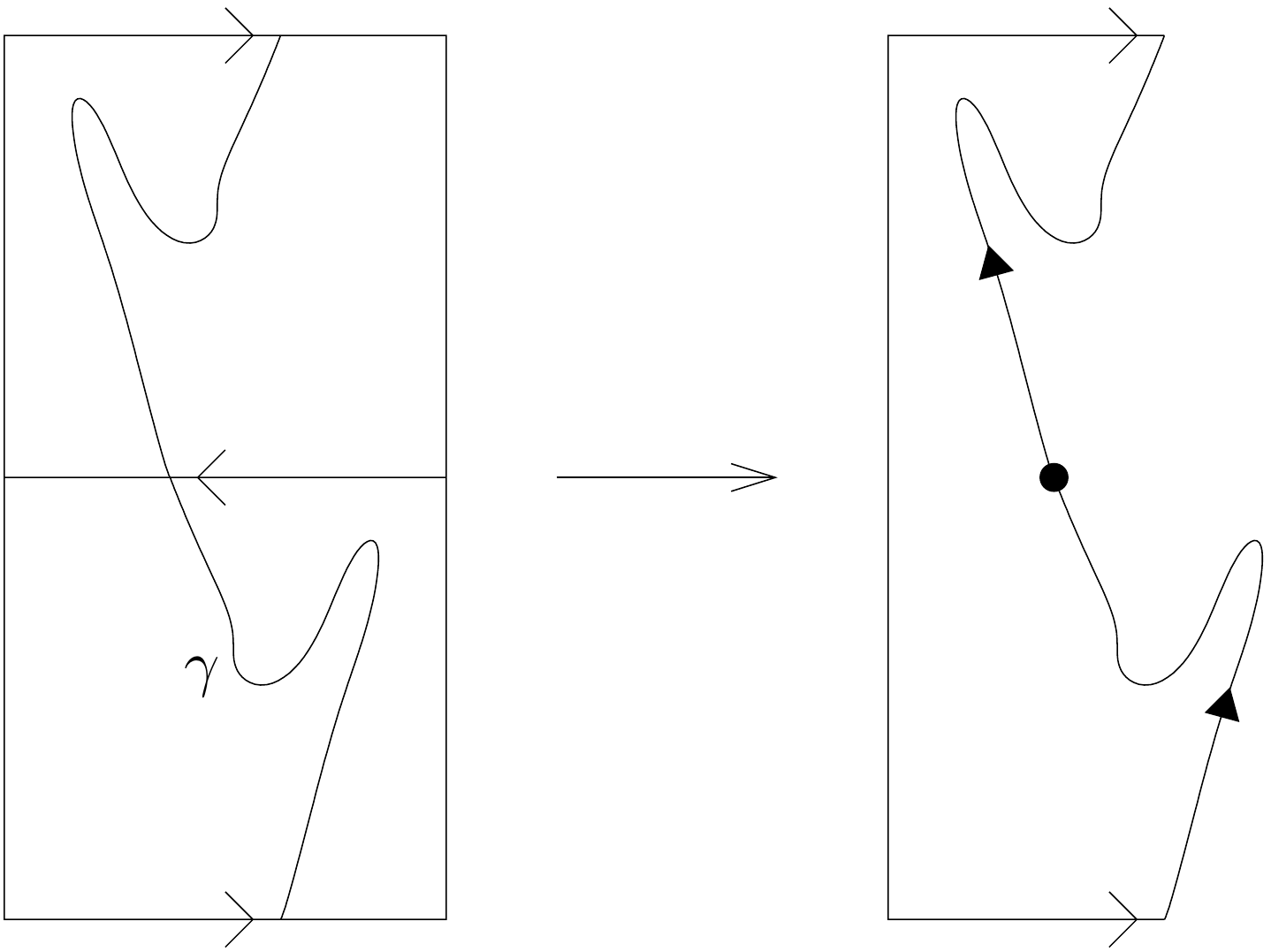}
\end{center}
\caption{Le découpage et recollage du ruban de Möbius.}\label{fig-mobius}
\end{figure}

\end{enumerate}

\bibliographystyle{smfalpha}
\bibliography{biblio}

\end{document}